\documentclass[12pt]{amsart}
\usepackage{amsmath,amssymb,amsbsy,amsfonts,amsthm,latexsym,mathabx,
            amsopn,amstext,amsxtra,euscript,amscd,stmaryrd,mathrsfs,
            cite,array,mathtools,enumerate}

\usepackage{url}
\usepackage[colorlinks,linkcolor=blue,anchorcolor=blue,citecolor=blue,backref=page]{hyperref}
\usepackage{color}
\begin{document}

 \renewcommand*{\backref}[1]{}
\renewcommand*{\backrefalt}[4]{%
    \ifcase #1 (Not cited.)%
    \or        (p.\,#2)%
    \else      (pp.\,#2)%
    \fi}

\newtheorem{theorem}{Theorem}
\newtheorem{lemma}[theorem]{Lemma}
\newtheorem{claim}[theorem]{Claim}
\newtheorem{cor}[theorem]{Corollary}
\newtheorem{prop}[theorem]{Proposition}
\newtheorem{definition}{Definition}
\newtheorem{question}[theorem]{Question}
\newcommand{\hh}{{{\mathrm h}}}

\numberwithin{equation}{section}
\numberwithin{theorem}{section}

\def\sssum{\mathop{\sum\!\sum\!\sum}}
\def\ssum{\mathop{\sum\ldots \sum}}

\def \balpha{\boldsymbol\alpha}
\def \bbeta{\boldsymbol\beta}
\def \bgamma{{\boldsymbol\gamma}}
\def \bomega{\boldsymbol\omega}

\newcommand{\Res}{\mathrm{Res}\,}

\def\sssum{\mathop{\sum\!\sum\!\sum}}
\def\ssum{\mathop{\sum\ldots \sum}}
\def\dsum{\mathop{\sum\  \sum}}
\def\iint{\mathop{\int\ldots \int}}

\def\squareforqed{\hbox{\rlap{$\sqcap$}$\sqcup$}}
\def\qed{\ifmmode\squareforqed\else{\unskip\nobreak\hfil
\penalty50\hskip1em\null\nobreak\hfil\squareforqed
\parfillskip=0pt\finalhyphendemerits=0\endgraf}\fi}%%

%  use the AMS-Euler Fraktur fonts
%%%%%%%%%%%%%%%%%%%%%%%%%%%%%%%%%%
\newfont{\teneufm}{eufm10}
\newfont{\seveneufm}{eufm7}
\newfont{\fiveeufm}{eufm5}
%%%%%%%%%%%%%%%%%%%%%%%%%%%%%%%%%
%
%  allow automatic size selection in math mode
%
%%%%%%%%%%%%%%%%%%%%%%%%%%%%%%%%%
\newfam\eufmfam
     \textfont\eufmfam=\teneufm
\scriptfont\eufmfam=\seveneufm
     \scriptscriptfont\eufmfam=\fiveeufm
%%%%%%%%%%%%%%%%%%%%%%%%%%%%%%%%%
%
%  \frak works on a single symbol at a time...
%
\def\frak#1{{\fam\eufmfam\relax#1}}

\def\fK{\mathfrak K}
\def\fT{\mathfrak{T}}

\def\rE{{\mathrm{E}}}
\def\rM{{\mathrm{M}}}

\def\fA{{\mathfrak A}}
\def\fB{{\mathfrak B}}
\def\fC{{\mathfrak C}}
\def\fD{{\mathfrak D}}
\def\fS{{\mathfrak S}}

\newcommand{\sX}{\ensuremath{\mathscr{X}}}

\def\eqref#1{(\ref{#1})}

\def\vec#1{\mathbf{#1}}
\def\dist{\mathrm{dist}}
\def\vol#1{\mathrm{vol}\,{#1}}

\def\squareforqed{\hbox{\rlap{$\sqcap$}$\sqcup$}}
\def\qed{\ifmmode\squareforqed\else{\unskip\nobreak\hfil
\penalty50\hskip1em\null\nobreak\hfil\squareforqed
\parfillskip=0pt\finalhyphendemerits=0\endgraf}\fi}

\def\sA{\mathscr A}
\def\sB{\mathscr B}
\def\sC{\mathscr C}
\def\sD{\Delta}
\def\sE{\mathscr E}
\def\sF{\mathscr F}
\def\sG{\mathscr G}
\def\sH{\mathscr H}
\def\sI{\mathscr I}
\def\sJ{\mathscr J}
\def\sK{\mathscr K}
\def\sL{\mathscr L}
\def\sM{\mathscr M}
\def\sN{\mathscr N}
\def\sO{\mathscr O}
\def\sP{\mathscr P}
\def\sQ{\mathscr Q}
\def\sR{\mathscr R}
\def\sS{\mathscr S}
\def\sU{\mathscr U}
\def\sT{\mathscr T}
\def\sV{\mathscr V}
\def\sW{\mathscr W}
\def\sX{\mathscr X}
\def\sY{\mathscr Y}
\def\sZ{\mathscr Z}

%%%%%%%%%%%%%%%%%%%%%%%%%
% Alphabet calligraphie %
%%%%%%%%%%%%%%%%%%%%%%%%%
\def\cA{{\mathcal A}}
\def\cB{{\mathcal B}}
\def\cC{{\mathcal C}}
\def\cD{{\mathcal D}}
\def\cE{{\mathcal E}}
\def\cF{{\mathcal F}}
\def\cG{{\mathcal G}}
\def\cH{{\mathcal H}}
\def\cI{{\mathcal I}}
\def\cJ{{\mathcal J}}
\def\cK{{\mathcal K}}
\def\cL{{\mathcal L}}
\def\cM{{\mathcal M}}
\def\cN{{\mathcal N}}
\def\cO{{\mathcal O}}
\def\cP{{\mathcal P}}
\def\cQ{{\mathcal Q}}
\def\cR{{\mathcal R}}
\def\cS{{\mathcal S}}
\def\cT{{\mathcal T}}
\def\cU{{\mathcal U}}
\def\cV{{\mathcal V}}
\def\cW{{\mathcal W}}
\def\cX{{\mathcal X}}
\def\cY{{\mathcal Y}}
\def\cZ{{\mathcal Z}}
\newcommand{\rmod}[1]{\: \mbox{mod} \: #1}

\def \wtcA {\widetilde \cA}
\def \wtcB {\widetilde \cB}
\def \wtcP {\widetilde \cP}

\def \wtA {\widetilde A}
\def \wtB {\widetilde B}
\def \wtP {\widetilde P}

\def\vr{\mathbf r}

\def\e{{\mathbf{\,e}}}
\def\ep{{\mathbf{\,e}}_p}
\def\em{{\mathbf{\,e}}_m}
\def\en{{\mathbf{\,e}}_n}

\def\Tr{{\mathrm{Tr}}}
\def\Nm{{\mathrm{Nm}}}

 \def\SS{{\mathbf{S}}}

\def\lcm{{\mathrm{lcm}}}

\def\({\left(}
\def\){\right)}
\def\fl#1{\left\lfloor#1\right\rfloor}
\def\rf#1{\left\lceil#1\right\rceil}

\def\mand{\qquad \mbox{and} \qquad}

\newcommand{\commA}[1]{\marginpar{%
\begin{color}{red}
\vskip-\baselineskip %raise the marginpar a bit
\raggedright\footnotesize
\itshape\hrule \smallskip A: #1\par\smallskip\hrule\end{color}}}

\newcommand{\commI}[1]{\marginpar{%
\begin{color}{blue}
\vskip-\baselineskip %raise the marginpar a bit
\raggedright\footnotesize
\itshape\hrule \smallskip I: #1\par\smallskip\hrule\end{color}}}

\newcommand{\commO}[1]{\marginpar{%
\begin{color}{magenta}
\vskip-\baselineskip %raise the marginpar a bit
\raggedright\footnotesize
\itshape\hrule \smallskip O: #1\par\smallskip\hrule\end{color}}}

%%%%%%%%%%%%%%%%%%%%%%%%%%%%%%%%%%%%%%%%%%%%%%%%%%%%%%%%
%%%%%%%%%%%%%%%%%%%%%%%%%%%%%%%%%%%%%%%%%%%%%%%%%%%%%%%%
%%%%%%%%%%%%%%%%%%%%%%%%%%%%%%%%%%%%%%%%%%%%%%%%%%%%%%%%
%%%%%%%%%%%%%%%%%%%%%%%%%%%%%%%%%%%%%%%%%%%%%%%%%%%%%%%%

%%%%%%%  END OF STANDARD STUFF %%%%%%%%%

%%%%%%%%%%%%%%%%%%%%%%%%%%%%%%%%%%%%%%%%%%%%%%%%%%%%%%%%
%%%%%%%%%%%%%%%%%%%%%%%%%%%%%%%%%%%%%%%%%%%%%%%%%%%%%%%%
%%%%%%%%%%%%%%%%%%%%%%%%%%%%%%%%%%%%%%%%%%%%%%%%%%%%%%%%
%%%%%%%%%%%%%%%%%%%%%%%%%%%%%%%%%%%%%%%%%%%%%%%%%%%%%%%
%%%%%%%%%%%
%%% Spell

\hyphenation{re-pub-lished}

\parskip 4pt plus 2pt minus 2pt
%% \parskip= 2 pt plus 3pt

%% \parindent 0 pt

%\mathsurround=1pt

\def\bfdefault{b}
\overfullrule=5pt

\def \F{{\mathbb F}}
\def \K{{\mathbb K}}
\def \Z{{\mathbb Z}}
\def \Q{{\mathbb Q}}
\def \R{{\mathbb R}}
\def \C{{\mathbb C}}
\def\Fp{\F_p}
\def \fp{\Fp^*}

\def\ABC{ab+ac+bc}
\def\sABC{\sC(\cA, \cB, \cC)}
\def\sAAA{\sC(\cA)}
\def\sIk{\sI_k(\cA, \cB, \cC)}
\def\SABC{S_\psi(\cA, \cB, \cC)}
\def\TABC{T_\chi(\cA, \cB, \cC)}
\def\fABC{\fS_{\chi,\psi}(\cA, \cB, \cC)}
\def\SAAA{S_\psi(\cA, \cB, \cC)}
\def\TAAA{T_\chi(\cA)}

\title[Balog--Wooley Decomposition and Convolutions]{Analogues of the Balog--Wooley Decomposition for Subsets of Finite Fields and 
 Character Sums with Convolutions}

\author[O. Roche-Newton]{Oliver  Roche-Newton }
\address{Institute of Financial Mathematics and Applied Number Theory, Johannes Kepler University Linz,
Altenberger Stra\ss e 69, A-4040 Linz, Austria}
\email{o.rochenewton@gmail.com }

\author[I. E.~Shparlinski]{Igor E. Shparlinski}

\address{School of Mathematics and Statistics, University of New South Wales\\
2052 NSW, Australia.}

\email{igor.shparlinski@unsw.edu.au}

\author[A. Winterhof]{Arne Winterhof}
\address{Johann Radon Institute for
Computational and Applied Mathematics,
Austrian Academy of Sciences, Altenberger Stra\ss e 69, A-4040 Linz, Austria}
\email{arne.winterhof@oeaw.ac.at}

\begin{abstract} 
 Balog and Wooley have recently proved that any subset $\cA$ of either real numbers or of a prime finite field can be decomposed into
two parts $\cU$ and $\cV$, one of small additive energy and the other of small multiplicative energy.  
In the case of arbitrary finite fields, we obtain an analogue that under some natural restrictions for a rational function $f$ both the additive energies of $\cU$ and $f(\cV)$ are small. Our method is based on bounds of character sums which 
leads to the restriction $\# \cA > q^{1/2}$ where $q$ is the field size. The bound is optimal, up to logarithmic factors, when $\# \cA  \geq q^{9/13}$.
Using $f(X)=X^{-1}$ we 
apply this result to estimate some triple additive and multiplicative  character sums involving three sets
with convolutions $\ABC$ with variables $a,b,c$ running through 
%coming from 
three arbitrary subsets of a finite field. 
%Given three sets $\cA, \cB, \cC$ in a finite field we study various additive and 
%multiplicative properties 
%of the  set $\{\ABC~:~a\in \cA, \, b \in \cB, \, c\in \cC\}$. 
\end{abstract}


\maketitle

\section{Introduction} 

\subsection{Background}

Let $\F_q$ denote the finite field of $q$ elements of characteristic $p$. 
%We use $\Psi$ and $\cX$ 
%to denote, respectively,  the sets of additive and multiplicative characters in $\F_q$, see~\cite{IwKow, LN}
%for some background on characters. Furthermore, we use  $\Psi^*$ and $\cX^*$ to denote the
%sets of nontrivial characters. 

Given two sets $\cU,
\cV \subseteq  \F_q$,  as usual, we define their sum  and product sets
as
%\begin{align*}
$$
\cU + \cV = \{u+v~:~ u\in\cU,\ v \in \cV\}
\quad \text{and}\quad
\cU \cdot \cV = \{uv~:~ u\in\cU,\ v \in \cV\}.
$$

The sum-product problem is concerned with proving that, for a given set $\cU$ in a field $\F$, at least one of $\cU+\cU$ and $\cU \cdot \cU$ has cardinality significantly larger than the original set $\cU$. This problem has been widely studied in the finite field setting in recent years, originating from the work of Bourgain, Katz and Tao~\cite{BKT} and subsequently Bourgain, Glibichuk and Konyagin~\cite{BGK} in proving that for some absolute constants $c, \varepsilon >0$, the bound
\begin{equation}
 \max \{\#(\cU+\cU), \#(\cU \cdot \cU )\} \geq c ( \# \cU)^{1+\varepsilon}
 \label{bkt}
 \end{equation}
holds for all $\cU \subseteq \mathbb \F_p$, subject to certain necessary restrictions on $\# \cU$. See~\cite{RNRS} for the best estimates for this problem and for further background on sum-product estimates.

A basic tool in sum-product estimates is the notion of different kinds of energy. The additive energy $\rE(\cU)$ of the set $\cU \subseteq \F_q$ 
is defined as
$$
\rE(\cU) =\# \{(u_1,u_2,u_3,u_4) \in \cU^4~:~u_1 + u_2 = u_3 + u_4\}.
$$
Note that the multiplicative energy, denoted by $\rE^\times(\cU)$, is defined similarly with respect to the equation 
$u_1   u_2 = u_3   u_4$.  It follows from a straightforward application of the Cauchy-Schwarz inequality that
$$\rE(\cU) \geq \frac{ (\#\cU)^4}{\# (\cU+ \cU)},$$
and so good upper bounds on the additive energy of $\cU$ translate into good lower bounds for the size of the sum set of $\cU$. Similarly, good upper bounds on the multiplicative energy of $\cU$ translate into good lower bounds for the size of the product set of $\cU$.

In the spirit of the sum-product problem, one may naively expect that an analogue of the inequality~\eqref{bkt} holds, and that at least one of $\rE(\cU)$ and $\rE^\times(\cU)$ must be small. This is not true, as can be seen by taking $\cU$ to be the union of an arithmetic progression and a geometric progression of the same size. 
However, Balog and Wooley~\cite{BalWool} have shown something of this nature when they proved (in both the Euclidean and finite field setting) that the set $\cU$ can be written as a union of $\cV$ and $\cW$ such that $\rE(\cV)$ and $\rE^\times(\cW)$ are both small. These results were improved quantitatively in~\cite{KS} and~\cite{RSS}.

\subsection{An analogue of the Balog-Wooley Theorem}

Our main result is a generalisation of the Balog-Wooley
 decomposition~\cite[Theorem~1.3]{BalWool}. 
 
 For real   $Z > 1$ we define the quantity  
\begin{equation}
\label{eq:MZ} 
\rM(Z)=  \min \left \{   \frac{q^{1/2} }{Z^{1/2}(\log Z)^{11/4}}  ,  \,   \frac{Z^{4/5}}{q^{2/5}(\log Z)^{31/10}}         \right \} .
\end{equation}

As usual, we use the expressions $F \ll G$, $G \gg F $ and $F=O(G)$ to
mean $|F|\leq cG$ for some constant $c>0$. If the constant $c$ depends on a parameter $k$, we write 
$F = O_k(G)$ or $F \ll_k G$.  We also write  $F(x)=o(G(x))$ as an equivalent to $\lim\limits_{x\to \infty} F(x)/G(x)=0$.
Throughout the   paper,  we always use:
$$
\# \cA = A, \qquad \# \cB = B, \qquad \# \cC = C. 
$$
We denote by $p$ the characteristic of $\F_q$. 
\begin{theorem} 
\label{thm:decomp}
 For any set $\cA \subseteq  \F_q$ and any rational function $f\in \F_q(X)$ of degree $k$ which is not of the form $g(X)^p-g(X)+\lambda X+\mu$, there exist disjoint sets $\cS, \cT \subseteq \cA$ such that 
$\cA=\cS \cup \cT$ and with
$$\max\{\rE(\cS),\rE(f(\cT)) \}\ll_k \frac{A^3}{\rM(A)}.
$$
\end{theorem}  

One may check that Theorem~\ref{thm:decomp}  
 is non-trivial when $A \ge q^{1/2+\varepsilon}$, 
for any fixed $\varepsilon> 0$. For comparison, note that the decomposition results in~\cite{BalWool} and~\cite{RSS} are applicable below this range. This is because the main tool in~\cite{BalWool} and~\cite{RSS} is a strong new point-plane incidence bound of Rudnev~\cite{Rud}, whereas our main tool is the Weil bound. It may be within reach to obtain a version of Theorem~\ref{thm:decomp} which is non-trivial for smaller sets by finding a way to apply new results in incidence theory, but we have been unable to do this in the present paper.

Theorem~\ref{thm:decomp} covers some particularly natural choices of functions such as $f(X)=X^{-1}$ and $f(X)=X^2$ for odd $q$ which have been seen in sum-product literature before. For example, for these two functions, it is known 
(see~\cite[Propositions~12 and~14]{AYMRS}) that
\begin{equation}
\# \{u+v^2 : u,v \in \cU \} \gg (\#\cU)^{11/10} \quad \mbox{if } \#\cU<p^{3/5}
\label{square}
\end{equation}
and 
\begin{equation}
\#\{u+v^{-1}: u,v\in \cU\}\gg (\#\cU)^{31/30}\quad \mbox{if }\#\cU<p^{5/8}
\label{inverse}
\end{equation}
(we remark that the size of $\cU$ is bounded in terms of the characteristic $p$ rather than of $q$).
The moral here is that a non-linear function $f$  destroys any additive structure that originally exists in a set. 
A version of Theorem~\ref{thm:decomp} with $\cA \subseteq \C$ and $f \in  \C(X)$ defined by $f(X)=X^{-1}$ has been given in~\cite[Theorem~9]{RSS}.

Note that the bounds~\eqref{square} and~\eqref{inverse} hold for smaller sets. This gives another hint that it may be possible to obtain a version of Theorem~\ref{thm:decomp} that gives a non-trivial bound below the square root threshold for certain special functions $f$.

The proof of Theorem~\ref{thm:decomp} is partly based on the work of Rudnev, Shkredov and Stevens~\cite{RSS}. 
We believe that it is of independent 
interest and may have several other applications. 

\subsection{The tightness of  the bound and conditions of Theorem~\ref{thm:decomp}}
Theorem~\ref{thm:decomp} becomes increasingly accurate as $A$ gets larger, and in fact is optimal up to logarithmic factors when $A \geq q^{9/13}$, as the following example over the prime field $\F_p$ illustrates. We consider the case when $f(X)=X^{-1}$, although a natural adaptation of the construction works for any rational function $f$ of degree $k$, with the construction getting slightly worse as $k$ increases. This is an adaptation of a construction from finite field sum-product theory, see Garaev~\cite[page~2736]{Gar2} for a presentation.

Let $\cJ$ be the interval $\{1,2,\dots, \lambda \}$, for an integer parameter $\lambda  < p$ to be chosen later, and then cover $\F_p$ by $\lceil p/ \lambda \rceil \ll p / \lambda$ disjoint intervals of size at most $\lambda$. By the pigeonhole principle, one of these intervals $\cJ_0$ has the property that $\#(\cJ_0 \cap f(\cJ)) \gg \lambda^2/p$. Define $\cA=\cJ_0 \cap f(\cJ)$. Then it follows that $\#(\cA+\cA) \ll \lambda \ll (Ap)^{1/2}$ and $\#(f(\cA)+f(\cA)) \ll \lambda \ll (Ap)^{1/2}$. For an arbitrary
function $f$ one has to work 
with the preimage  $f^{-1}(\cJ)$ of $\cJ$ 
(note that  $f(X)=X^{-1}$ we have  $f^{-1}=f$ and thus $f(\cJ)= f^{-1}(\cJ)$). 

Now, we can apply Theorem~\ref{thm:decomp} to this set, obtaining a decomposition $\cA=\cS \cup \cT$. One of these sets has cardinality at least $A/2$, and without loss of generality we assume that $S = \#\cS \geq A/2$. Then, by the Cauchy-Schwarz inequality,
$$A^4 \ll S^4 \leq \#(\cS+\cS) \rE(\cS) \leq \#(\cA+\cA)  \rE(\cS) \ll (Ap)^{1/2} \rE(\cS) ,$$
and so 
$$\rE(\cS) \gg \frac{A^3}{(p/A)^{1/2}}.$$
%% When $A \geq p^{9/13}$ 
When $A \gg  p^{9/13}$ (that is, for any choice of  $\lambda \gg p^{11/13}$)  
we have $\rM(A) \gg  (p/A)^{1/2}(\log A)^{-11/4}$, hence 
this lower bound matches the upper bound given by Theorem~\ref{thm:decomp}, up to logarithmic factors.

The following example shows that the condition on $f$ in
Theorem~\ref{thm:decomp} is needed:
let $\cA$ be any subset of $\F_q$ of additive energy $\rE(\cA)\gg A^3$ such as
an arithmetic progression or an additive subgroup. Then by the forthcoming
Lemma~\ref{sumdecomp} we have
$\max\{\rE(\cS),\rE(\cT)\}\gg A^3$ for any decomposition $\cA=\cS\cup \cT$ of
$\cA$. Now if $f(X)=\sum a_i X^{p^i}\in \F_q[X]$ is a linearized permutation
polynomial, and thus of the form
$g(X)^p-g(X)+\lambda X+\mu$, we have $f(a)+f(b)=f(a+b)$ and so
$$\max\{\rE(\cS),\rE(f(\cT))\}=\max\{\rE(\cS),\rE(\cT)\}\gg A^3.$$

\subsection{Applications to character sums}
We use $\Psi$ and $\cX$ 
to denote, respectively,  the sets of additive and multiplicative characters in $\F_q$, see~\cite{IwKow, LN}
for some background on characters. Furthermore, we use  $\Psi^*$ and $\cX^*$ to denote the
sets of nontrivial characters. 

Given three sets $\cA, \cB, \cC \subseteq \F_q$ we define
the following sums of  additive and multiplicative characters
$$
\SABC = \sum_{a\in \cA} \sum_{b \in \cB} \sum_{c\in \cC} \psi(\ABC), \qquad \psi \in \Psi,
$$
and 
$$
\TABC = \sum_{a\in \cA} \sum_{b \in \cB} \sum_{c\in \cC} \chi(\ABC), \qquad \chi \in \cX.
$$
Our interest to these sums is motivated by recent progress in bounds of additive and multiplicative 
character sums involving three sets, see~\cite{BouGar, PetShp} and~\cite{BalWool, Han, ShkShp}, 
respectively. 

For $\psi \in \Psi^*$ and $\chi \in \cX^*$ we have  the classical bounds of double sums
$$\max\left\{\left|\sum_{b\in \cB,\, c\in \cC}\psi(bc)\right|, \left|\sum_{b\in \cB,\, c\in \cC}\chi(bc+1)\right|\right\}
= O (\sqrt{BCq}),
$$
see, for example,~\cite[Corollaries~1 and~5]{gysa08},
where the constant in the symbol ``$O$'' is absolute and can be easily evaluated, 
(in fact it can be taken as 1 for additive character sums and also for multiplicative character sums
if  $\cB \subseteq \F_q^*$ or  $\cC \subseteq \F_q^*$). 
These 
immediately yield the bounds 
\begin{equation}
\label{eq:bilin1}
\SABC = O\(A\sqrt{BCq}\) 
\end{equation}
and, provided $ 0\not\in \cA$, 
\begin{equation}
\label{eq:bilin2}
\TABC= O\(A\sqrt{BCq}\). 
\end{equation} 

%see, for example,~\cite[Equation~(1.4)]{BouGar} and~\cite[Equation~(1.4)]{ShkShp}, respectively. 

Note that~\eqref{eq:bilin1} and~\eqref{eq:bilin2} are best possible in general.
For example, take $q=r^2$, for a prime power $r$, then it is easy to check that 
with $\cA=\F_r$, resp. $\cA=\F_r^*$, $\cB=\cC=\F_r$ and any  $\psi \in \Psi^*$ and  $\chi \in \cX^*$  which are trivial on $\F_r$, 
these bounds are attained. 
%%, that is, $\psi(x)=\psi_1(\alpha x)$
%%for any of the $r-1$ different $\alpha\in \F_q^*$ of $\F_r$-trace $0$, where $\psi_1$ denotes the canonical additive character. 
%Hence,  
%$\SABC=ABC=r^3=\sqrt{ABq}C$.
%
%Similarly, we can take $\cA=\cB=\cC=\F_r$ and any non-trivial multiplicative character $\chi$ of $\F_q$  
%which is trivial on $\F_r^*$. Then $\TABC=r^3-r^2$ and $(\sqrt{A}+1)\sqrt{Bq}C+AB=r^3+r^{5/2}+r^2$ are of the same order of magnitude $r^3$.

For additive character sums we also provide an example when $q=p$ is prime. Take $\cA=\cB=\cC=\{0,1,2,\ldots, \lfloor 0.1p^{1/2}\rfloor\}$. Then we have
$0\le ab+ac+bc \le 0.03p$ for any $a,b,c\in \cA\times \cB\times \cC$ and thus for the additive canonical character $\psi(x)=\cos(2\pi x/p)+i\sin(2\pi x/p)$ of $\F_p$ we get
$|\SABC|\ge ABC \cos(0.06 \pi)\ge 0.98 ABC$ which is of the same order of magnitude $p^{3/2}$ as $A\sqrt{BCp}$. 

However, if, say, $\cC$ is a sufficiently large structured set, we can get improvements. For example, if $\cC$ is an additive subgroup of $\F_q$,
we get
$$\SABC\le A\sum_{b\in \F_q}\left|\sum_{c\in \cC}\psi(bc)\right|\le A q$$
by~\cite[Lemma~3.4]{wi01}, which improves~\eqref{eq:bilin1}   if $q=o(BC)$, as well as,
$$
\TABC\le \min\{A,B\}C+ABq^{1/2}
$$
by~\cite{wi01} (Remark~(iii) below Theorem~3.7), which improves~\eqref{eq:bilin2} if $B = o(C)$.
Similar results can be obtained for other structured sets $\cC$ such as arithmetic or 
geometric progressions.
If the sets $\cA$ and $\cB$ are also structured, one can take further advantage of this.  

In passing, we note that, sum-product and incidence theory can be used to show that
$$
\sABC = \{\ABC~:~a\in \cA, \, b \in \cB, \, c\in \cC\}
$$
 is always large. A recent result 
of  Pham, Vinh and de Zeeuw~\cite{PVdZ} gives a lower bound if $A=B=C$,  
$$\# \sC(\cA, \cB, \cC) \gg \min \{ p, A^{3/2} \}.$$
There is  little doubt that it can be extended to cover the case when the
variables come from sets $\cA, \cB, \cC$ of different sizes.

Although we have not been able to improve~\eqref{eq:bilin1} and~\eqref{eq:bilin2}, 
we obtain several related results about the structure of the set $\sABC$.
For example, we use Theorem~\ref{thm:decomp} with $f(X)=X^{-1}$, to show that any sufficiently large set $\cB \subseteq \F_q$ contains a large subset $\cW$  such that for any set $\cA \subseteq \F_q$ 
%either $S_\psi(\cW, \cB, \cC)$ or $T_\chi(\cW, \cB, \cC)$ can be estimated nontrivially. 
at least one of $S_\psi(\cA, \cW, \cW)$ or $T_\chi(\cA, \cW, \cW)$ can be estimated nontrivially.

\begin{theorem}
\label{thm:S or T}
 For any sets $\cA,\cB \subseteq  \F_q^*$  there exists a subset $\cW \subseteq \cB$ 
 of cardinality $W \ge B/2$ such that for any 
% sets $\cB, \cC \subseteq \F_q$  with   $\# \cB= B$
%$\# \cC=C$ and 
characters $ \(\chi,   \psi\) \in \Psi^* \times \cX^*$ we
 have
$$
\min \left\{ |S_\psi(\cA, \cW, \cW)|, |T_\chi(\cA, \cW, \cW)|\right\} \ll  A^{1/2}B^{3/2}q^{1/2}\rM(B)^{-1/2}, 
$$
where $\rM(Z)$ is defined by~\eqref{eq:MZ}.
\end{theorem} 

We can prove a weaker bound for three possibly different sets $\cA,\cB,\cC$.
\begin{theorem}
 \label{thm:AnotB}
 For any sets $\cA,\cB,\cC\subseteq \F_q^*$ there exist subsets $\cW_1\subseteq \cB$ and $\cW_2\subseteq \cC$
 of cardinalities $W_1\ge B/2$ and $W_2\ge C/2$ such that for any characters $\(\chi,\psi\)\in \Psi^*\times \cX^*$ we have
 $$\min \left\{ |S_\psi(\cA, \cW_1, \cW_2)|, |T_\chi(\cA, \cW_1, \cW_2)|\right\} \ll  \frac{A^{1/2}(BC)^{3/4}q^{1/2}}{\max\{\rM(B),\rM(C)\}^{1/4}},$$
 where $\rM(Z)$ is defined by~\eqref{eq:MZ}.
\end{theorem}

Finally, we present a bound for a mixed sum of multiplicative and additive characters. Define 
$$
\fABC = \sum_{a\in \cA} \sum_{b \in \cB} \sum_{c\in \cC}  \chi(\ABC) \psi(\ABC),\quad \(\chi,   \psi\) \in \Psi \times \cX.
$$
In the case when both $\chi$ and $\psi$ are nontrivial, we obtain a  bound which makes an effective use 
of all three variables. 
 
 \begin{theorem}
\label{thm:Mixed}
 For any sets $\cA,\cB, \cC \subseteq  \F_q^*$  %with  $\# \cA = A$,  $\# \cB= B$ and
%$\# \cC=C$ 
and characters  $ \(\chi,   \psi\) \in \Psi^* \times \cX^*$ we
 have
$$
\fABC \ll  (ABCq)^{1/2} + A^{1/2} BC q^{1/4}. 
$$
\end{theorem}  

Note that Theorem~\ref{thm:Mixed} is non-trivial provided that 
$$q = o(\min\{ABC,A^2\}),$$ and takes the form 
$\fABC \ll  (ABCq)^{1/2}$ for $BC \le q^{1/2}$.

We also give an application of Theorem~\ref{thm:decomp}
to bilinear sums with incomplete Kloosterman sums over arbitrary sets. 
In fact this result is motivated by, and somewhat mimics,  the argument
of Balog and Wooley~\cite[Section~6]{BalWool}, which in turn is based on a low energy 
decomposition~\cite[Theorem~1.3]{BalWool}.
Namely, given  sets $\cA,\cB, \cC \subseteq  \F_q^*$ and three sequences of complex weights 
$\balpha = (\alpha_a)_{a \in \cA}$, $\bbeta = (\beta_b)_{b \in \cB}$
and $\bgamma = (\gamma_c)_{c \in \cC}$, we define
$$
K(\cA,\cB,\cC; \balpha, \bbeta, \bgamma)
= \sum_{a\in \cA} \sum_{b \in \cB}\alpha_a \beta_b
\left| \sum_{c \in \cC} \gamma_c  \psi\(ac + bc^{-1}\)\right|^2.
$$
As usual we use $\|\balpha\|_\sigma$ to denote the $L_\sigma$-norm of the weights $\balpha$, 
see~\eqref{eq:norm} below.
%
%Using the standard approach, that is. applying the Cauchy-Schwarz inequality and the extending 
%the ranges of summation over $a$ and $b$ to the whole field, for any $\psi \in \Psi^*$, one obtains
%\begin{equation} 
%\label{eq:K triv}
%\begin{split}
%K(\cA,\cB,\cC; \balpha, \bbeta, \bgamma) & \ll  \|\balpha\|_2 \|\bbeta\|_2   \|\bgamma\|_\infty^2 q C \\
%& \le \|\balpha\|_\infty \|\bbeta\|_\infty   \|\bgamma\|_\infty^2 qA^{1/2}B^{1/2} C .
%\end{split}
%\end{equation}

 \begin{theorem}
\label{thm:Kloost}
 For any sets $\cA,\cB, \cC \subseteq  \F_q^*$, complex weights 
 $\balpha = (\alpha_a)_{a \in \cA}$, $\bbeta = (\beta_b)_{b \in \cB}$, 
 $\bgamma = (\gamma_c)_{c \in \cC}$
 and a character  $   \psi  \in \Psi^*$, we
 have
$$
K(\cA,\cB,\cC; \balpha, \bbeta, \bgamma)  \ll  \(  \|\balpha\|_1 \|\bbeta\|_2  +  \|\balpha\|_2 \|\bbeta\|_1\)   \|\bgamma\|_\infty^2  q^{1/2} C^{3/2}\rM(C)^{-1/2},
$$
where $\rM(Z)$ is defined by~\eqref{eq:MZ}.
\end{theorem} 

Writing the bound of Theorem~\ref{thm:Kloost} in terms of $\|\balpha\|_\infty$ and $\|\bbeta\|_\infty$, we derive
$$
K(\cA,\cB,\cC; \balpha, \bbeta, \bgamma)  \ll 
 \|\balpha\|_\infty \|\bbeta\|_\infty   \|\bgamma\|_\infty^2 \frac{q^{1/2}(AB^{1/2}+A^{1/2} B)  C^{3/2}}{\rM(C)^{1/2}}.
$$
This is nontrivial provided 
$$
\min\{A,B\} C \ge \frac{q^{1+\varepsilon}}{\rM(C)}
$$
for some fixed $\varepsilon>0$,
which improves the range $\max\{A,B\} C \ge q^{1+\varepsilon}$ which can be achieved
by the bound $\|\balpha\|_\infty \|\bbeta\|_\infty \|\bgamma\|_\infty^2 \min\{A,B\}Cq$ obtained via the standard approach
if $A$ and $B$ are of the same order of magnitude and $C$ satisfies the above inequality.

We prove these bounds in Section~\ref{sec:char sum}.

\subsection{Notation}

Usually we use capital letters in italics to denote sets and the same capital letters 
in roman to denote their cardinalities, as in the following example $\#\cX = X$. In  particular, 
as we have mentioned,  we always do this for the sets   $\cA, \cB, \cC$.  

The notation 
$\cA=\cU \sqcup \cV$ is used for the {\it partition\/} of $\cA = \cU \cup \cV$ in a union of disjoint sets $\cU\cap \cV=\emptyset$.

Let $f\not\equiv 0$ be a rational function on $\F_q$. We can express $f$ as a quotient $f=g/h$, where $g$ and $h\not\equiv 0$ 
are coprime polynomials. 
The degree of $f$ is defined as $\max \{ \deg g, \deg h \}$. Note that, for a rational function $f$ of degree $k \geq 1$ and any $c \in  \F_q$, we have $\# \{x~ :~ f(x)=c \} \leq k$.

Given two sets $\cU,  \cV \subseteq \F_q^*$, a rational function 
$f\in \F_q(X)$ and an element $a \in \F_q$, we use $r_{\cU,\cV}(f;a)$ 
to denote the number of solutions to the equation $f(u)+f(v) = a$, $(u,v) \in \cU \times \cV$. 
Furthermore, we simplify it in two special cases by writing 
$r_{\cU}(f;a)$ if $\cU = \cV$ and $r_{\cU,\cV}(a)$ if $f(X) = X$; and use $r_{\cU}(a)$ when
both. This notation is used with flexibility; for example $r_{\cU,-\cV}(a)$ denotes the number of solutions to the equation $u-v=a$ with $(u,v) \in \cU \times \cV$.

The letters  $k$, $m$ and $n$ (in both the upper and
lower cases) denote positive integer 
numbers. 

We define the norms of a complex sequence $\balpha=(\alpha_m)_{m\in \cI}$ for some finite set $\cI$ of indices by
\begin{equation}
\label{eq:norm}
 \|\balpha\|_\infty=\max_{m\in \cI}|\alpha_m|  \mand \|\balpha\|_\sigma =\( \sum_{m\in \cI} |\alpha_m|^\sigma\)^{1/\sigma},
\end{equation} 
where  $\sigma >0$. %and similarly for the weights $\bbeta$ and $\bgamma$. 

 \section{Preliminary results}

% and equations}

%
%
%As we have mentioned, we need the following bound on double character sums which is essentially
%due to Karatsuba~\cite{Kar1}, see also~\cite[Chapter~VIII, Problem~9]{Kar2}, 
%which can easily be derived  from the
%Weil bound (see~\cite[Corollary~11.24]{IwKow}) and the H{\"o}lder inequality.
%We present it in the form for hybrid sums, which   involves both additive 
%and multiplicative characters. Since in this form it does not seem to have appeared in 
%the literature, we supply a concise proof. 
%
%We call a triple $(\psi, \chi, k) \in \Psi\times \cX \times \Z$ {\it nontrivial\/}
%if  $k > 0$ and at least one of the characters $\psi$ or $\chi$ is nontrivial, 
%or if  $k \le 0$ and $\chi$ is nontrivial
%
%

We  need the following bound on mixed character sums, which we derive from the Weil bound.

\begin{lemma}
\label{lem:AddMult} Take characters $ \(\chi,   \psi\) \in  \cX\times \Psi^*$. For any  rational function
$f\in \F_q(X)$ of degree $k$ and not of the form $f(X) = g(X)^p-g(X)+ \lambda X + \mu$ if  $\chi$ is trivial,
we have
$$
\sum_{u\in \cU}\sum_{v \in \cV}  \chi(u+v)  \psi\(f(u+v)\) 
 \ll_k\sqrt{UVq} ,
$$
where %the implied constant depends only on $k$ and 
we use the convention that   $\psi(f(x))=0$ if $x$ is a pole of $f$.
\end{lemma}

\begin{proof} Denote
$$
\Sigma = \sum_{u\in \cU}\sum_{v \in \cV}  \chi(u+v)  \psi\(f(u+v)\) . 
$$
By the orthogonality
 of additive characters 
\begin{equation*} 
\begin{split}
\Sigma &= \sum_{x\in \F_q}  \chi(x) \psi\(f(x)\) \frac{1}{q} \sum_{\lambda \in \F_q}
 \sum_{u\in \cU}\sum_{v \in \cV}  
\psi(\lambda(  u+v-x))\\
 &= \frac{1}{q} \sum_{\lambda \in \F_q} \sum_{x\in \F_q}  \chi(x) \psi\(f(x)-\lambda x\)
  \sum_{u\in \cU} \psi(\lambda u  ) \sum_{v \in \cV}   
\psi(\lambda v).
\end{split}
\end{equation*}
By  the Weil bound, see~\cite[Appendix~5, Example~12]{Weil},  and the assumption of non-linearity 
of $f$, 
the  sum over $x$ is $O_k(q^{1/2})$.
Thus 
$$
\Sigma \ll_k q^{-1/2}  \sum_{\lambda \in \F_q} \left|
  \sum_{u\in \cU} \psi(\lambda u  )  \right| \left| \sum_{v \in \cV}   
\psi(\lambda v)\right|.
$$
Using  the Cauchy-Schwarz inequality and then again 
orthogonality of additive characters, we derive
\begin{equation*} 
\begin{split}
  \sum_{\lambda \in \F_q}  \left|
  \sum_{u\in \cU} \psi(\lambda u  )  \right| & \left| \sum_{v \in \cV}   
\psi(\lambda v)\right|\\
& \le \( \sum_{\lambda \in \F_q}  \left|
  \sum_{u\in \cU} \psi(\lambda u  )  \right|^2\)^{1/2} \( \sum_{\lambda \in \F_q}   \left| \sum_{v \in \cV}   
\psi(\lambda v)\right|^2\)^{1/2}\\
& = (q^2 UV)^{1/2}
\end{split}
\end{equation*}
and the result follows. 
\end{proof}

%We will need the following result, which follows from a result of
% Hart, Iosevich and Solymosi~\cite[ Theorem~1.5]{HaIoSo}.

We  need the following result, which bounds the number of solutions to certain equations over $\F_q$.
\begin{lemma}
 \label{lem:prodsum}  Suppose that $\cW, \cX, \cY, \cZ \subseteq \F_q$. % with cardinalities $W$, $X$, $Y$, $Z$, respectively. 
 For any %fixed integer  $k$ 
 %and any 
 rational function
$f\in \F_q(X)$ of degree $k$ and  not of the form $f(X) = g(X)^p-g(X)+ \lambda X + \mu$, 
 the number $J$ of solutions to the equation
$$f(w+x) = y+z  \qquad (w,x,y,z) \in \cW\times \cX \times \cY \times \cZ$$
satisfies the bound
$$
J \leq \frac{WXYZ}{q} +  O_k((WXYZq)^{1/2}) .
$$
\end{lemma}

\begin{proof}
Using  the orthogonality of  additive  characters, we write
$$
J = \sum_{(w,x,y,z) \in \cW\times \cX \times \cY \times \cZ} \frac{1}{q} \sum_{\psi\in \Psi}\psi\(f(w+x)-y -z\).
$$
Rearranging the terms and separating the contribution from the trivial character, we obtain
$$
J - \frac{WXYZ}{q}  \ll  \frac{1}{q} \sum_{\psi\in \Psi^*}
\left| \sum_{(w,x) \in \cW\times \cX} 
\psi\(f(w+x)\) \right| \left| \sum_{y \in  \cY} \psi(y) \right| \left| \sum_{z \in  \cZ} \psi(z) \right|.
$$
By Lemma~\ref{lem:AddMult} with the trivial multiplicative character, we have
$$
J - \frac{WXYZ}{q}  \ll_k   \frac{\sqrt{WXq}}{q}   \sum_{\psi\in \Psi^*}
\left| \sum_{y \in  \cY} \psi(y) \right| \left| \sum_{z \in  \cZ} \psi(z) \right|.
$$
By the Cauchy-Schwarz inequality as in the proof of Lemma~\ref{lem:AddMult}
we obtain 
$$
 \sum_{\psi\in \Psi^*}
\left| \sum_{y \in  \cY} \psi(y) \right| \left| \sum_{z \in  \cZ} \psi(z) \right| \le (q^2 YZ)^{1/2}
$$
and the result follows. 
\end{proof}

We now use Lemma~\ref{lem:prodsum} to study the multiplicities of elements in the sum set of 
values of $f$. 
The names of the variables in  Lemma~\ref{lem:shkredov}
are chosen to match those in  Lemma~\ref{lemma:main}, where it is applied. 

%Namely, given two sets $\cU,  \cV \subseteq \F_q^*$ and elements $a \in \F_q$, we use $r_{U,V}(a)$ 
%to denote the number of solutions to the equation $u-v = a$, $(u,v) \in \cU \times \cV$. 

 \begin{lemma} \label{lem:shkredov} Let $\cA, \cS, \cU \subseteq \F_q^*$.
 %, with 
%$A=\# \cA$,  $S = \# \cS$, $U = \# \cU$. 
Let   $u>0$ be such that $r_{\cS,-\cA}(x) \geq u$ for all $x \in \cU$. 
Let $k$ be a fixed positive integer and suppose also that 
$$
\tau \geq 2 \frac{kASU}{uq}.
$$
Then,  for any rational function
$f\in \F_q(X)$, of degree $k$ and  not of the form $f(X) =  g(X)^p-g(X)+\lambda X + \mu$, we have 
 $$\# \{x\in \F_q~:~r_{\cU}(f; x) \geq \tau \}  \ll_k \frac{A S U  q }{u^2\tau^2} .$$
\end{lemma}

\begin{proof}
Define
$$
\cR = \{x\in \F_q~:~r_{\cU}(f; x) \geq \tau \} .
$$
Note that for $R =\# \cR$ we have 
$$
\tau R  \le \sum_{x \in \cR}r_{\cU}(f; x) =\# \{(x,y,z) \in  \cR \times \cU  \times \cU~:~x = f(y)+f(z) \}.
$$
On the other hand, since $r_{\cS,-\cA}(z)  \geq u$ for $z\in \cU$,  we have
\begin{align*}
&\# \{(x,y,z) \in  \cR \times \cU  \times \cU~:~x = f(y)+f(z)  \}\\
& \qquad  \le u^{-1}  \# \{(v,w, x,y) \in  \cS \times \cA \times \cR \times \cU~:~x= f(y)+f(v-w)\}.
\end{align*}
Therefore, 
\begin{align*}
\tau u R  &\le  \# \{(v,w, x,y) \in  \cS \times \cA \times \cR \times \cU~:~x= f(y)+f(v-w)\}
\\ & \leq k \cdot \# \{(v,w, x,\widetilde y) \in  \cS \times \cA \times \cR \times f(\cU)~:~x= \widetilde y+f(v-w)\}.
\end{align*}
Applying Lemma~\ref{lem:prodsum}, 
it follows that
$$ \tau u R \le \frac{kARSU}{q}  + O_k((ARSUq)^{1/2}).$$
The assumed lower bound on $\tau$ then implies that
$$\tau u R \ll_k (ARSUq)^{1/2}$$
and the result follows.
\end{proof}

We need the following result~\cite[Lemma~17]{RSS}. We include a short proof for the convenience of the reader, which 
can be easily extended to finite sets in any group.

\begin{lemma} \label{sumdecomp}  Let $\cA_1,\ldots,\cA_n \subseteq \F_q$. Then
$$ \rE \( \bigcup_{i=1}^n \cA_i \) \leq \( \sum_{i=1}^n\rE^{1/4}(\cA_i)\)^4.$$
\end{lemma}

\begin{proof} We can assume that the sets $\cA_1,\ldots,\cA_n$ are disjoint. 
Then we have
 \begin{align*}
 \rE\(\bigcup_{i=1}^n\cA_i\)
&=\sum_{i,j,k,\ell=1}^n \sum_{x\in \F_q} r_{\cA_i,\cA_j}(x) r_{\cA_k,\cA_\ell}(x)\\
&\le \sum_{i,j,k,\ell=1}^n\(\sum_{x\in \F_q} r_{\cA_i,\cA_j}(x)^2\)^{1/2}
\(\sum_{x\in \F_q}r_{\cA_k,\cA_\ell}(x)^2\)^{1/2}\\
&= \(\sum_{i,j=1}^n\(\sum_{x\in \F_q} r_{\cA_i,\cA_j}(x)^2\)^{1/2}\)^2\\
&= \(\sum_{i,j=1}^n\(\sum_{x\in \F_q} r_{\cA_i,-\cA_i}(x)r_{\cA_j,-\cA_j}(x)\)^{1/2}\)^2\\
&\le  \(\sum_{i,j=1}^n\(\sum_{x\in \F_q} r_{\cA_i,-\cA_i}(x)^2\right)^{1/4} \left(\sum_{x\in\F_q}r_{\cA_j,-\cA_j}(x)^2\)^{1/4}\)^2\\
&=  \(\sum_{i=1}^n\(\sum_{x\in \F_q} r_{\cA_i,-\cA_i}(x)^2\)^{1/4}\)^4
=  \(\sum_{i=1}^n \rE(A_i)^{1/4}\)^4
\end{align*}
via two consecutive  applications of the Cauchy-Schwarz inequality.
\end{proof}

We now formulate and prove our main technical tool. 

\begin{lemma} \label{lemma:main} Let $\cA \subseteq \F_q$.
Then  for any rational function
$f\in \F_q(X)$ of degree $k$ and not of the form $f(X) =  g(X)^p-g(X)+\lambda X + \mu$,  
there exists $\cU \subseteq  \cA$
of cardinality $U$ such that 
$$U \gg \frac{\rE^{1/2}(\cA)}{A^{1/2} (\log A)^{7/4}}$$ 
and
$$
\rE(f(\cU)) \ll_k   \frac{A U^6q^{-1}(\log A)^{11/2} + A U^3 q(\log A)^{6}}{\rE(\cA)}.
$$
\end{lemma}

\begin{proof}
%First we remark that since we trivially have $\rE(\cA)\le A^3$, the second  term contributes to 
%$\rE(f(\cU)) $ at least $A^{-1} U^3 q^{1/2}(\log A)^{6}$. Hence the result is trivial unless 
%we have 
%\begin{equation}
%\label{eq:large A}
%A \ge q^{1/2}. 
%\end{equation} 

The additive energy of a set $\cA \subseteq \F_q$ can be written in the form
$$\rE(\cA)=\sum_{x \in \cA+\cA} r_{\cA,\cA}^2(x).$$
Dyadically decompose this sum, and deduce that there is a popular dyadic set
$$\cS=\{x \in \cA+\cA~:~\rho \leq r_{\cA,\cA}(x)<2\rho \}$$
with some integer $1\le \rho \le |A|$ where $\rho$ is a power of $2$, and such that 
\begin{equation}
\label{diadenergy}
\rho^2 \# \cS \gg \frac{\rE(\cA)} {\log A}.
\end{equation} 
Consider the point set
$$\cP=\{(a,b) \in \cA \times \cA ~:~a+b \in \cS \}.$$
Following our standard convention, we denote
$$
P = \# \cP \mand S = \# \cS
$$
and note that 
\begin{equation}
\label{eq:SPS}
\rho S \le  P < 2\rho S.
\end{equation}  
We then make a second dyadic decomposition of this point set to find 
a large subset supported on vertical lines with approximately the same richness.

To be precise, for any $x \in  \F_q$, define 
$$
\cA_x=\{y~:~(x,y) \in \cP \} \mand A_x = \# \cA_x.
$$
Note that
$$
\sum_{x \in \cA} A_x =P.
$$
Therefore, for some $s$ there exists a dyadic set
$$\cV =\{x \in \cA~:~s \leq A_x<2s \}
$$
such that, recalling~\eqref{eq:SPS}, for  $V = \# \cV$ we have
\begin{equation}
\label{eq:VsS}
V s \gg \frac{P}{\log A}  \gg \frac{\rho  S  }{ \log A}.
\end{equation}

We now separate into two cases 
$$
V  \geq   \frac{s}{(\log A)^{1/2}} \mand V < \frac{s}{(\log A)^{1/2}}.
$$
 
%\subsubsection*{Case~1: $V  \geq s / \log A$}  
\subsubsection*{Case~I: $V  \geq s(\log A)^{-1/2}$}  
Note for any  $x \in \cV$,   there exist 
$$
y_1,y_2, \ldots,y_s \in \cA_x \subseteq \cA
$$ 
such that $(x,y_i) \in \cP$ for all $1 \leq i \leq s$. Therefore
$$x+y_1,x+y_2,\ldots,x+y_s \in \cS.$$
It follows that $r_{\cS,-\cA}(x) \geq s$ for every $x \in \cV$ and 
in this case we define 
\begin{equation}
\label{eq:Case1}
\cU = \cV \mand u=s.
\end{equation}

%\subsubsection*{Case~2: $V  <s / \log A$} 
\subsubsection*{Case~II: $V  <s(\log A)^{-1/2}$} 
In this case, consider the point set 
$$\cQ=\{(x,y) \in \cP~:~x \in \cV \} 
$$
of cardinality $Q = \# \cQ$. 
As before, we note that  for any  $x \in \cV$,   there exist at least $s$ values of $y \in  \cA_x \subseteq \cA$
with $(x,y)\in \cP$. Hence  $Q \geq V s$.

 Now, for any $y \in  \F_q$, define 
$$
\cB_y=\{x ~:~(x,y) \in \cQ \} \mand B_y = \# \cB_y.
$$
Note that
$$\sum_{y \in \cA} B_y =Q.$$
Therefore, for some $t$ there exists a dyadic set
$$\cW=\{y \in \cA ~: ~t \leq B_y<2t \}$$
such that for $W = \# \cW$ we have
\begin{equation}
\label{eq:WVs}
Wt   \gg \frac{Q}{\log A} \geq  \frac{V s}{\log A}.
\end{equation}
Note  that since $ \cQ \subseteq \cV \times \cA$ we also have $t \leq V$.
 It follows from~\eqref{eq:WVs} and the assumption that $s > V(\log A)^{1/2}$ that
$$WV \geq  Wt \gg     \frac{V s}{\log A} >  \frac{V^2}{(\log A)^{1/2}}$$
and thus
\begin{equation}
\label{eq:W large}
W \gg   V (\log A)^{-1/2} \geq t (\log A)^{-1/2} .
\end{equation}
Also, by~\eqref{eq:WVs} and~\eqref{eq:VsS},
\begin{equation}
\label{eq:Wt large}
Wt   \gg   \frac{V s}{\log A} \gg  \frac{\rho S}{(\log A)^2}.
\end{equation}

Now, let $y \in \cW$. So, there exist $x_1,x_2, \ldots,x_{t} \in \cA$ such that $(x_i,y) \in \cP$ for all $1 \leq i \leq t$. Therefore
$$x_1+y,x_2+y,\ldots,x_{t}+y \in \cS.$$
It follows that $r_{\cS,-\cA}(y) \geq t$ for every $y \in \cW$.

In this case we take  
\begin{equation}
\label{eq:Case2}
\cU = \cW \mand u=t.
\end{equation}

One now verifies that for both choices~\eqref{eq:Case1} and~\eqref{eq:Case2} 
we have $\cU \subseteq A$ of cardinality $U$ with 
\begin{equation}
\label{eq2}
U \gg  u (\log A)^{-1/2}
\end{equation}
and
\begin{equation}
\label{eq3}
 u U  \gg \frac{\rho  S  }{(\log A)^2}
\end{equation}
and such that
$$r_{\cS,-\cA}(x) \geq u, \qquad \forall x \in \cU.$$
Indeed in Case~I, the inequality~\eqref{eq2} is by the assumption, while in  
Case~II this follows from~\eqref{eq:W large}.
Furthermore, in Case~I, the inequality~\eqref{eq3} is weaker than~\eqref{eq:VsS}, while in  
Case~II this follows from~\eqref{eq:Wt large}.

Note also that, multiplying~\eqref{eq2} and~\eqref{eq3} and then using~\eqref{diadenergy} 
together with the fact that $\rho \leq A$, we obtain
\begin{equation}
\label{eq4}
U^2    \gg \frac{\rho S } {(\log A)^{5/2}} = \frac{\rho^2  S } {\rho(\log A)^{5/2}} \gg  \frac{\rE(\cA) }{  A(\log A)^{7/2}}.
\end{equation}
This $\cU$ is the desired set. It remains to estimate  the energy $\rE(f(\cU))$ of $f(\cU)$.

We have 
\begin{equation}
\label{eq:E and r2}
\rE(f(\cU))= \sum_{x \in \F_q} r_{f(\cU)}(x)^2 \le 
 % \ll_k 
 \sum_{x \in \F_q} r_{\cU}(f;x)^2. 
\end{equation}
Define the set  
$$
\cR_0=\left\{x \in \F_q~:~r_{\cU}(f;x) \leq 2 \frac{k A S U}{uq}\right\}
$$
and then for $J = \rf{\log A/\log 2}$, the sets
$$
\cR_j=\left\{x \in \F_q~:~ 2^{j} \frac{k A S U}{uq} < r_{\cU}(f;x) \le 2^{j+1} \frac{k A S U}{uq}\right\}, 
\quad j = 1, \ldots, J.
$$
Since
$$
\sum_{x \in \F_q} r_{\cU}(f;x) = U^2
$$
the contribution from $x \in \cR_0$ is 
bounded by 
\begin{equation}
\label{eq:R0}
 \sum_{x \in \cR_0} r_{\cU}(f;x)^2 \le  
  2 \frac{k A S U}{uq}\sum_{x \in \F_q} r_{\cU}(f;x)  \ll  \frac{k A S U^3}{uq}.
\end{equation}
For $j = 1, \ldots, J$, we apply Lemma~\ref{lem:shkredov} 
%with $\nu = 1$ (which is justified by the assumption~\eqref{eq:large A}) and 
with 
$$
\tau = 2^{j} \frac{k A S U}{uq} 
$$
 to derive
\begin{equation}
\label{eq:Rj}
 \sum_{x \in \cR_j} r_{\cU}(f;x)^2\le (2\tau)^2 \# \cR_j \ll_k
  \frac{A S  U  q }{u^2}.
\end{equation}

Substituting the bounds~\eqref{eq:R0} and~\eqref{eq:Rj} in~\eqref{eq:E and r2} and 
using the fact that $J %\ll \log q 
\ll \log A$, we obtain
\begin{equation}
\label{eq:E 2Term}
\rE(f(\cU))
\ll_k   \frac{A S U^3}{uq} +  \frac{A S  U  q }{u^2}\log A. 
\end{equation}

We now deal with these two terms in~\eqref{eq:E 2Term} one at a time. 

Firstly, multiplying~\eqref{eq3} with the first inequality  in~\eqref{eq4}, and then recalling~\eqref{diadenergy},
we obtain 
$$
uU^3 \gg \frac{\rho^2S^2}{(\log A)^{9/2}}  \gg  \frac{S\rE(\cA)}{(\log A)^{11/2}} 
$$
which we rewrite as 
$$
  \frac{S}{u} \ll  \frac{U^3 (\log A)^{11/2}}{\rE(\cA)}.
$$
Hence, for the first term in~\eqref{eq:E 2Term} we have
\begin{equation}
\label{eq:Term1}
 \frac{A S U^3}{uq} \ll   \frac{A U^6(\log A)^{11/2}}{\rE(\cA)q}.
\end{equation}

For the second term, squaring~\eqref{eq3} and then using~\eqref{diadenergy} again, we obtain
$$
u^2U^2 \gg  \frac{\rho^2S^2}{(\log A)^{4}}  \gg  \frac{S\rE(\cA)}{(\log A)^{5}} 
$$
which we rewrite as 
$$
  \frac{S}{u^2} \ll  \frac{U^2 (\log A)^{5}}{\rE(\cA)}.
$$
Hence, for the second term in~\eqref{eq:E 2Term} we have
\begin{equation}
\label{eq:Term2}
 \frac{A S  U  q }{u^2}\log A\ll   \frac{A U^3 q(\log A)^{6}}{\rE(\cA)}.
\end{equation}

Substituting~\eqref{eq:Term1} and~\eqref{eq:Term2} in~\eqref{eq:E 2Term}, 
we conclude the proof. 
\end{proof}

\section{Proof of Theorem~\ref{thm:decomp}}

\subsection{Partition procedure}
Below, we use  Lemma~\ref{lemma:main} to construct a nested sequence of sets
$$\emptyset=\cU_1 \subseteq \cU_2 \subseteq \ldots \subseteq \cU_m$$
and a corresponding sequence
$$\cV_m \subseteq \cV_{m-1} \subseteq \ldots \subseteq \cV_1=\cA$$
where $\cU_i  \sqcup \cV_i=\cA$, $i =1, \ldots, m$. 

When 
\begin{equation}
\label{eq:TermCond}
\rE(\cV_m) \leq A^3 / \rM(A)
\end{equation}
we terminate this process and take $\cS=\cV_m$ and $\cT=\cU_m$ (as we explain below,~\eqref{eq:TermCond}
is satisfied for some $m$).

The nested sequence is defined as follows. Suppose that 
\begin{equation}
\label{eq:NoTermCond}
\rE(\cV_i) > A^3/\rM(A).
\end{equation}
 Then, by Lemma~\ref{lemma:main}, there exists $\cQ_{i} \subseteq \cV_i$ such that for $Q_i = \# \cQ_i$
\begin{equation}
\label{eq:fact1}
Q_i \gg \frac{\rE^{1/2}(\cV_i)}{A^{1/2}\log^{7/4} A} > \frac{A} { \rM(A)^{1/2} \log^{7/4} A}
\end{equation}
and using~\eqref{eq:NoTermCond} we derive
\begin{eqnarray}\label{EQi} \rE(f(\cQ_i)) &
\ll_k & \frac{V_i Q_i^6q^{-1}(\log A)^{11/2} + V_i Q_i^3 q(\log A)^{6}}{\rE(\cV_i)}
\\\nonumber & < &\frac{\rM(A)}{A^3} \left ( \frac{V_i Q_i^6 (\log A)^{11/2}}{q}+V_i Q_i^3q (\log A)^6 \right).
\end{eqnarray}
In particular, inequality~\eqref{eq:fact1} implies that
\begin{equation}
\label{eq:fact2}
V_i \ll Q_i \rM(A)^{1/2} \log^{7/4} A.
\end{equation}
Then, define $\cV_{i+1}=\cV_i \setminus \cQ_i$. This automatically defines $\cU_{i+1}=\cU_i \sqcup \cQ_i$. This iterative construction in fact gives 
\begin{equation}
\label{eq:U and Q}
\cU_{i+1}= \bigsqcup_{j=1}^{i} \cQ_j.
\end{equation}
Note that this process  certainly terminates, since we have a uniform lower bound on the cardinality of $Q_i$ and so the cardinality 
$V_i$  is monotonically decreasing, thus we eventually reach the termination 
condition~\eqref{eq:TermCond}.

\subsection{Final estimate}
Recall that by~\eqref{eq:U and Q} we have
$$
\cT = \cU_m =  \bigsqcup_{j=1}^{m-1} \cQ_j.
$$
Therefore, Lemma~\ref{sumdecomp} and the inequality~\eqref{EQi} imply  that 
\begin{align*}
& \rE^{1/4}(f(\cT))=\left (\rE \left (\bigcup_{i=1}^{m-1} f(\cQ_i) \right )\right )^{1/4} \leq \sum_{i=1}^{m-1}\rE^{1/4}(f(\cQ_i))
\\& \qquad  \ll_k \sum_{i=1}^{m-1}\left ( \frac{\rM(A) }{A^3} \left ( \frac{V_i Q_i^6(\log A)^{11/2}}{q}+V_iQ_i^3q (\log A)^6 \right) \right )^{1/4}.
\end{align*}
Now the inequality~\eqref{eq:fact2} yields
\begin{align*}
& \rE^{1/4}(f(\cT))\\& \quad  \ll_k  
\sum_{i=1}^{m-1}\left ( \frac{\rM(A) }{A^3} \left ( \frac{V_i Q_i^6(\log A)^{11/2}}{q}+\rM(A)^{1/2}Q_i^4q (\log A)^{31/4} \right) \right )^{1/4}.
\end{align*}

Using that $A \ge V_i \ge Q_i$ and thus replacing $V_i Q_i^6$ with $A^3  Q_i^4$  in the first 
term, we now obtain
\begin{align*}
& \rE^{1/4}(f(\cT)) \\
& \qquad  \ll_k  \( \frac{\rM(A) (\log A)^{11/2}}{q} +A^{-3} \rM(A)^{3/2}q (\log A)^{31/4} \) ^{1/4}.
\sum_{i=1}^{m-1}Q_i. 
\end{align*}
Using that 
$$
\sum_{i=1}^{m-1}Q_i \le A
$$
we obtain 
\begin{equation}
\label{eq:bound 2}
\rE(f(\cT)) \ll_k  \frac{\rM(A) A^4 (\log A)^{11/2}}{q} + \rM(A)^{3/2} A q (\log A)^{31/4}.
\end{equation}
We now see that the choice~\eqref{eq:MZ} balances between~\eqref{eq:TermCond} and~\eqref{eq:bound 2}
and leads to the inequality 
$$
\rE(f(\cT))   \ll_k A^3 / \rM(A)
$$
implied by our choice $\cS = \cV_m$ and the bound~\eqref{eq:TermCond}.
This completes the proof.~\hfill $\Box$

\section{Proofs of Theorems~\ref{thm:S or T}, \ref{thm:AnotB}, \ref{thm:Mixed} and~\ref{thm:Kloost}}
\label{sec:char sum}

\subsection{Character sums and energy}

First we give some basic bounds of the sums $\SABC$ and $\TABC$ in terms of energies of various sets. 

\begin{lemma}
\label{lem:AddCharEnergy} For any sets $\cA,\cB, \cC \subseteq  \F_q$   and additive 
character  $ \psi\ \in \Psi^* $ we
 have
$$
\SABC \le   A^{1/2} \rE(\cB)^{1/4} \rE(\cC)^{1/4} q^{1/2}.
$$
\end{lemma}

\begin{proof} By the Cauchy-Schwarz inequality we have
\begin{align*}
|\SABC|^2 & \le  A \sum_{a\in \cA} \left|\sum_{b \in \cB} \sum_{c\in \cC}  \psi(\ABC)\right|^2\\
& \le  A \sum_{a\in \F_q} \left|\sum_{b \in \cB} \sum_{c\in \cC}   \psi(\ABC)\right|^2\\
&= A \sum_{b_1, b_2 \in \cB} \sum_{c_1, c_2\in \cC}   \psi\(b_1c_1- b_2c_2\) \\
& \qquad \qquad \qquad \qquad \sum_{a\in \F_q} \psi\(a\(b_1+c_1-b_2-c_2\)\).
\end{align*}
By the orthogonality of additive characters 
we obtain 
$$
\SABC \le  \sqrt{A\rE(\cB,\cC)q}, 
$$
where 
$$
\rE(\cB,\cC) = \#\{(b_1, b_2,c_1, c_2)\in\cB \times \cB \times \cC \times\cC ~:~b_1+c_1=b_2+c_2\}.
$$

Finally, it follows from the Cauchy-Schwarz inequality that 
$$\rE(\cB,\cC) \le \rE(\cB)^{1/2}\rE(\cC)^{1/2}.$$ 
Indeed,
\begin{align*}
\rE(\cB,\cC)&=\sum_{x \in \F_q} r_{\cB,-\cB}(x) r_{\cC,-\cC}(x)
\\&\leq \left (\sum_{x \in \F_q} r_{\cB,-\cB}^2(x) \right)^{1/2} \cdot \left (\sum_{x \in \F_q} r_{\cC,-\cC}^2(x) \right)^{1/2}
\\&=\rE(\cB)^{1/2}\rE(\cC)^{1/2},
\end{align*}
which completes the proof.
%Again, using the orthogonality of additive characters, we 
%write
%\begin{align*}
%\rE(\cB,\cC) & =  \sum_{b_1, b_2 \in \cB} \sum_{c_1, c_2\in \cC}
%\frac{1}{q} \sum_{\xi \in \Psi} \xi\(b_1+c_1- b_2-c_2\)  \\
%& = 
%\frac{1}{q} \sum_{\xi \in \Psi}  \left| \sum_{b \in \cB}\xi(b)\right|^2 \left| \sum_{c \in \cC}
%\xi(c)\right|^2\\
%&\le \(\frac{1}{q} \sum_{\xi \in \Psi}  \left| \sum_{b \in \cB}\xi(b)\right|^4\)^{1/2}
% \(\frac{1}{q} \sum_{\xi \in \Psi}  \left| \sum_{c \in \cC} \xi(c)\right|^4\)^{1/2}
 %= \sqrt{\rE(\cB) \rE(\cC)}
%\end{align*} 
%and the result follows. 
\end{proof} 

We also have an analogue of Lemma~\ref{lem:AddCharEnergy}
for multiplicative characters.

\begin{lemma}
\label{lem:MultCharEnergy} For any sets $\cA,\cB, \cC \subseteq  \F_q^*$   and character  $ \psi\ \in \Psi^* $ we
 have
$$
\TABC \ll   A^{1/2} \rE(\cB^{-1})^{1/4} \rE(\cC^{-1})^{1/4} q^{1/2}+A^{1/2}BC.
$$
\end{lemma}

\begin{proof} By the Cauchy-Schwarz inequality we have
\begin{align*}
|\TABC|^2 & \le  A \sum_{a\in \cA} \left|\sum_{b \in \cB} \sum_{c\in \cC}  \chi(\ABC)\right|^2\\
& \le  A \sum_{a\in \F_q^*} \left|\sum_{b \in \cB} \sum_{c\in \cC}  \chi(\ABC) \right|^2\\
&= A \sum_{b_1, b_2 \in \cB} \sum_{c_1, c_2\in \cC}  \chi\(\frac{b_1c_1}{b_2c_2}\) \\
& \qquad \qquad   \sum_{a\in \F_q^*}  \chi\(a\(b_1^{-1}+c_1^{-1}\)+1\) 
\overline\chi\(a\(b_2^{-1}+c_2^{-1}\)+1\)\\
&= A \sum_{b_1, b_2 \in \cB} \sum_{c_1, c_2\in \cC}  \chi\(\frac{b_1c_1}{b_2c_2}\) \\
& \qquad \qquad   \sum_{a\in \F_q^*}  \chi\(a^{-1} + b_1^{-1}+c_1^{-1}\)
\overline\chi\(a^{-1} + b_2^{-1}+c_2^{-1}\), 
\end{align*}
where $\overline\chi$ is the complex conjugate character. 
Making the change of variable $a \to a^{-1}$ in the sum over $a \in \F_q^*$
we arrive to 
\begin{align*}
|\TABC|^2 
&\le A \sum_{b_1, b_2 \in \cB} \sum_{c_1, c_2\in \cC}  \chi\(\frac{b_1c_1}{b_2c_2}\) \\
& \qquad \qquad   \sum_{a\in \F_q^*}  \chi\(a + b_1^{-1}+c_1^{-1}\)
\overline\chi\(a + b_2^{-1}+c_2^{-1}\).
\end{align*}
By the ``approximate'' orthogonality of multiplicative  characters, that is, by 
$$
 \sum_{a\in \F_q^*}  \chi\(a +u\) 
\overline\chi\(a+v\)
\ll 
\left\{
\begin{array}{ll}
1 & \text{if } u \ne v,  \\
  q & \text{if } u = v ,
\end{array}
\right.
$$
we obtain 
$$
\TABC \ll  \sqrt{A\rE(\cB^{-1},\cC^{-1})q}  + A^{1/2} BC , 
$$
where 
$$
\rE(\cB^{-1},\cC^{-1}) = \#\{(b_1, b_2,c_1, c_2)\in\cB  \times \cB \times \cC \times\cC ~:~
b_1^{-1}+c_1^{-1}=b_2^{-1}+c_2^{-1}\}.
$$
As in the proof of  Lemma~\ref{lem:AddCharEnergy} we obtain
$$
\rE(\cB^{-1},\cC^{-1}) \le   \sqrt{\rE(\cB^{-1}) \rE(\cC^{-1})}
$$
and the result follows. 
\end{proof}

\subsection{Concluding the proofs} 

We are now ready to establish the desired results. 

\begin{proof}[Proof of Theorem~\ref{thm:S or T}]
 By Theorem~\ref{thm:decomp} there exist $\cS$ and $\cT$ with additive energies
$$\max\{\rE(\cS),\rE(\cT^{-1})\}\ll \frac{B^3}{\rM(B)}$$
and $\cB=\cS \cup \cT$. Hence, one of these sets contains at least $B/2$ elements. Let $\cW$ be this set. 
Lemmas~\ref{lem:AddCharEnergy} and~\ref{lem:MultCharEnergy} 
complete the proof. Note that the second summand in Lemma~\ref{lem:MultCharEnergy} is smaller than the bound of Theorem~\ref{thm:S or T}.
\end{proof}

\begin{proof}[Proof of Theorem~\ref{thm:AnotB}] 
Similarly to the proof of Theorem~\ref{thm:S or T}, there exist $\cS_1,\cS_2,\cT_1,\cT_2$ with 
$$\max\{\rE(\cS_1),\rE(\cT_1^{-1})\}\ll \frac{B^3}{\rM(B)},$$
$$\max\{\rE(\cS_2),\rE(\cT_2^{-1})\}\ll \frac{C^3}{\rM(C)},$$
$\cB=\cS_1\cup \cT_1$ and $\cC=\cS_2\cup \cT_2$. Let $\cW_i$, $i=1,2$, be the larger set of $\cS_i$ and $\cT_i$. Then the result follows again by Lemmas~\ref{lem:AddCharEnergy} and~\ref{lem:MultCharEnergy},
which completes the proof of Theorem~\ref{thm:AnotB}.
\end{proof}

%\begin{cor}
%\label{cor:}
%For any fixed 
%$\varepsilon>0$, there exists $\delta > 0$ such that  
% 
%\end{cor}

%We also consider the mixed sums
%of multiplicative and additive characters 
%$$
%\fABC = \sum_{a\in \cA} \sum_{b \in \cB} \sum_{c\in \cC}  \chi(\ABC) \psi(\ABC),\quad \(\chi,   \psi\) \in \Psi \times \cX,
%$$
%and in the case when both $\chi$ and $\psi$ are nontrivial, we obtain a  bound which makes an effective use 
%of all three variables. 

\begin{proof}[Proof of Theorem~\ref{thm:Mixed}] 
The proof follows a combination of the calculations of Lemmas~\ref{lem:AddCharEnergy}
and~\ref{lem:MultCharEnergy}. 
By the Cauchy-Schwarz inequality we have
\begin{align*}
|\fABC|^2 &  \le  A \sum_{a\in \cA} \left|\sum_{b \in \cB} \sum_{c\in \cC}  \chi(\ABC) \psi(\ABC)\right|^2\\
& \le  A \sum_{a\in \F_q^*} \left|\sum_{b \in \cB} \sum_{c\in \cC}  \chi(\ABC) \psi(\ABC)\right|^2\\
&= A \sum_{b_1, b_2 \in \cB} \sum_{c_1, c_2\in \cC}  \chi\(\frac{b_1c_1}{b_2c_2}\)  \psi\(b_1c_1- b_2c_2\)\\
& \qquad   \quad \sum_{a\in \F_q^*}    \chi\(a^{-1}+ b_1^{-1}+c_1^{-1}\) 
\overline\chi\(a^{-1} + b_2^{-1}+c_2^{-1}\)\\
& \qquad  \qquad  \qquad \qquad  \qquad \qquad  \psi\(a\(b_1+c_1-b_2-c_2\)\).
\end{align*}
Since both $\chi$ and $\psi$ are nontrivial characters we see that the inner sum over $a$ 
satisfies the Weil bound, see~\cite[Appendix~5, Example~12]{Weil}, unless 
$$
b_1^{-1}+c_1^{-1}= b_2^{-1}+c_2^{-1} \mand b_1+c_1 = b_2+c_2.
$$
Clearly, when $(b_2, c_2) \in \cB \times \cC$ are fixed (in $BC$ ways) the pair $(b_1, c_1)$ 
can take at most two values. 
Hence, we obtain 
$$ 
|\fABC|^2  \ll  A  (B^2C^2 q^{1/2} + BC q)
$$
and the result follows. \end{proof}

\begin{proof}[Proof of Theorem~\ref{thm:Kloost}] 
We now partition the set $\cC$ into two sets $\cS$ and $\cT$ as in Theorem~\ref{thm:decomp}
with respect  to the function $f(X) = X^{-1}$. Thus
$$
\max\{\rE(\cS),\rE(\cT^{-1})) \}\ll \frac{C^3}{\rM(C)}.
$$
%By the Cauchy-Schwarz inequality
Since $(x+y)^2\le 2(x^2+y^2)$ for any real $x$ and $y$, we obtain 
%\begin{align*}
%K(\cA,\cB,\cC; \balpha, \bbeta, \bgamma) & \ll 
% \sum_{a\in \cA} \sum_{b \in \cB}|\alpha_a| |\beta_b|
%\left| \sum_{c \in \cS} \gamma_c  \psi\(ac + bc^{-1}\)\right|^2\\
%&\qquad \quad  +  \sum_{a\in \cA} \sum_{b \in \cB}|\alpha_a| |\beta_b|\left| \sum_{c \in \cT} \gamma_c  \psi\(ac + bc^{-1}\)\right|^2.
%\end{align*}
%Therefore, 
\begin{equation}
\label{eq:K}
K(\cA,\cB,\cC; \balpha, \bbeta, \bgamma)   \le 2 \(\sum_{a \in \cA}   |\alpha_a| U_a+   \sum_{b \in \cB}  |\beta_b| V_b\), 
\end{equation}
where 
$$
U_a  = \sum_{b \in \cB}  |\beta_b|
\left| \sum_{c \in \cT} \gamma_c  \psi\(ac + bc^{-1}\)\right|^2,~
V_b = \sum_{a\in \cA}   |\alpha_a|
\left| \sum_{c \in \cS} \gamma_c  \psi\(ac + bc^{-1}\)\right|^2.
$$

By the Cauchy-Schwarz inequality, as before, for every $a \in \cA$ we derive
\begin{align*}
U_a^2 & \ll  \|\bbeta\|_2^2
\sum_{b\in \cB}\left| \sum_{c \in \cT} \gamma_c  \psi\(ac + bc^{-1}\)\right|^4\\
& \le   \|\bbeta\|_2^2
\sum_{b\in \F_q}\left| \sum_{c \in \cT} \gamma_c  \psi\(ac + bc^{-1}\)\right|^4\\
 & \le   \|\bbeta\|_2^2  \|\bgamma\|_\infty^4 q \rE(\cT^{-1}) \ll  \|\bbeta\|_2^2  \|\bgamma\|_\infty^4 q C^3\rM(C)^{-1}.
\end{align*}
Similarly, for every $b \in \cB$, we have
$$
V_b^2 \ll  \|\balpha\|_2^2  \|\bgamma\|_\infty^4 q \rE(\cS) 
\ll  \|\balpha\|_2^2  \|\bgamma\|_\infty^4 q   C^3\rM(C)^{-1}.
$$
Substituting these bounds in~\eqref{eq:K} we conclude the proof. 
\end{proof}

\section{Comments}

It is easy to verify that at the cost of merely typographical changes one can 
obtain a full analogue of Theorem~\ref{thm:decomp} for 
$$
\rE(f(\cT), g(\cT)) = \# \{t_1, t_2, t_3, t_4 \in \cT^4~:~f(t_1) + g(t_2) = f(t_3) + g(t_4)\}.
$$
Most likely, using multiplicative character sums instead of additive character sums, 
one can obtain similar results for $\max\{\rE^\times(\cS),\rE(f(\cT)) \}$ or even
 $\max\{\rE^\times(\cS), \rE(f(\cT), g(\cT)) \}$. However the following question 
 appears to require new ideas. 
 
 \begin{question}
 Given two bivariate polynomials $F(X,Y), G(X,Y) \in \F_q[X,Y]$, satisfying some natural 
 conditions, show that any set $\cA\subseteq \F_q$ can be partitioned as $\cA=\cS \sqcup \cT$
 in such a way that 
 the two sets
 $$
 \{\(s_1, s_2, s_3, s_4\) \in \cS^4~:~F(s_1,s_2) = F(s_3,s_4)\}
 $$
 and
 $$
 \{\(t_1, t_2, t_3, t_4\) \in \cT^4~:~G(t_1,t_2) = G(t_3,t_4)\}
$$
are both of small cardinality. 
\end{question}

 Finally, we note that Theorem~\ref{thm:Kloost} can be extended to  the sums
\begin{equation}
\label{eq:ugly}
\sum_{a\in \cA} \sum_{b \in \cB} \sum_{c\in \cC}  \sum_{d\in \cD}
\psi\(a(c+d) + b(f(c) + g(d))\)
\end{equation}
over sets $\cA, \cB, \cC, \cD \subseteq \F_q$ and rational functions $f,g \in \F_q(X)$. 
Within our approach,  one can show that, under some natural conditions on $f$ and $g$,  
for any $\varepsilon > 0$ there are some $\delta>0$ and $\kappa > 0$, such 
that as long as $A(CD)^{1/2} , B(CD)^{1/2} \ge q^{1-\delta}$ and $C,D \ge q^{1/2 + \varepsilon}$ the  
sums~\eqref{eq:ugly} are of order at most $ABCD q^{-\kappa}$.  Despite a somewhat exotic shape 
of the sums~\eqref{eq:ugly}, they  
may be used in the theory of randomness extractors in arbitrary  finite fields where the theory falls 
far below its counterpart in prime fields, see~\cite{B-SG,B-SK,BDZ} for more details
and further references.   One can also introduce
weights of the form $\alpha_a$, $\beta_b$ and $\gamma_{c,d}$ in the sums~\eqref{eq:ugly}.

%
%
%Finally, we note that very similarly to the application of~\cite[Theorem~1.3]{BalWool} to 
%bounds of quadruple multiplicative character sums,  the decomposition result of Theorem~\ref{thm:decomp} 
%can be applied to quadruple additive  character sums, for example to the sums
%\begin{equation}
%\label{eq:ugly}
%\sum_{a\in \cA} \sum_{b \in \cB} \sum_{c\in \cC}  \sum_{d\in \cD}
%\psi\(a(c+d) + b(f(c) + g(d))\)
%\end{equation}
%over sets $\cA, \cB, \cC, \cD \subseteq \F_q$ and rational functions $f,g \in \F_q(X)$. 
%
%Within this approach one can show that, under some natural conditions on $f$ and $g$,  
%for any $\varepsilon > 0$ there are some $\delta>0$ and $\kappa > 0$, such 
%that as long as $A(CD)^{1/2} , B(CD)^{1/2} \ge q^{1-\delta}$ and $C,D \ge q^{1/2 + \varepsilon}$ the  
%sums~\eqref{eq:ugly} are of order at most $ABCD q^{-\kappa}$.  Despite a somewhat exotic shape 
%of the sums~\eqref{eq:ugly}, they  
%may be used in the theory of randomness extractors in arbitrary  finite fields where the theory falls 
%far below its counterpart in prime fields, see~\cite{B-SG,B-SK,BDZ} for more details
%and further references. Furthermore, in the special case $f(X) = g(X) = X^{-1}$ and $\cD = - \cC$ these sums take the form
%$$
%\sum_{a\in \cA} \sum_{b \in \cB} \sum_{c, d\in \cC}  \psi\(a(c - d) + b(c^{-1}  - d^{-1}\)
%= \sum_{a\in \cA} \sum_{b \in \cB}\left| \sum_{c \in \cC}   \psi\(ac + bc^{-1}\)\right|^2
%$$
%of average values of incomplete Kloosterman sums over arbitrary sets. One can also introduce
%weights of the form $\alpha_a$, $\beta_b$ and $\gamma_{c,d}$ in the sums~\eqref{eq:ugly}.

 \section*{Acknowledgements}
 
The authors thank Brendan Murphy, Misha Rudnev and Ilya Shkredov for helpful conversations. 

The first and third author were supported by the Austrian Science Fund FWF Projects F5509 and F5511-N26, respectively, which are  part of the 
Special Research Program ``Quasi-Monte Carlo Methods: Theory and Applications''.  
The second author
was supported  by the Australian Research Council Grant DP140100118.

\end{document}